\documentclass[11pt,a4paper]{amsart}
\usepackage[a4paper,top=3cm,bottom=2.5cm,left=3cm,right=3cm,marginparwidth=1.75cm]{geometry}
\usepackage[T1]{fontenc}
\usepackage[utf8]{inputenc}
\usepackage{lmodern}
\usepackage{amsmath,amssymb,amsthm,mathtools,mathrsfs}
\usepackage{booktabs,array,enumitem}
\usepackage{microtype,needspace,xurl}
\usepackage[dvipsnames]{xcolor}
\usepackage[colorlinks=true,linkcolor=blue,citecolor=blue,urlcolor=blue]{hyperref}
\usepackage{bookmark}
\newtheorem{theorem}{Theorem}[section]
\newtheorem{proposition}[theorem]{Proposition}
\newtheorem{lemma}[theorem]{Lemma}

\theoremstyle{definition}

\newtheorem{remark}[theorem]{Remark}
\numberwithin{equation}{section}

\newcommand{\xra}{\xrightarrow}
\newcommand{\lra}{\longrightarrow}

\let\emptyset\varnothing

\newcommand{\D}{\mathrm{D}}

\newcommand{\ZZ}{\mathbb{Z}}

\newcommand{\QQ}{\mathbb{Q}}

\newcommand{\CC}{\mathbb{C}}
\newcommand{\PP}{\mathbb{P}}
\newcommand{\FF}{\mathbb{F}}

\newcommand{\RR}{\mathbb{R}}

\newcommand{\ch}{\mathrm{ch}}
\newcommand{\reg}{\mathrm{reg}}

\newcommand{\HH}{\mathbb{H}}

\newcommand{\pr}{\mathrm{pr}}

\newcommand{\Ram}{\mathrm{Ram}}

\renewcommand{\Re}{\operatorname{Re}}
\renewcommand{\Im}{\operatorname{Im}}

\newcommand{\cX}{\mathcal{X}}

\newcommand{\cB}{\mathcal{B}}
\newcommand{\cK}{\mathcal{K}}
\newcommand{\cI}{\mathcal{I}}

\newcommand{\cP}{\mathcal{P}}

\newcommand{\rH}{\mathrm{H}}
\newcommand{\cO}{\mathcal{O}}

\newcommand{\tpi}{\tilde{\pi}}

\DeclareMathOperator{\Def}{Def}
\DeclareMathOperator{\CH}{CH}

\DeclareMathOperator{\tr}{tr}
\DeclareMathOperator{\id}{id}

\newcommand{\Mon}{\operatorname{Mon}}

\setlist[enumerate,1]{label=\textup{(\alph*)},leftmargin=*,itemsep=2pt,beginpenalty=10000}
\setlist[itemize]{leftmargin=*,itemsep=2pt}
\allowdisplaybreaks[1]
\usepackage{tikz-cd}
\usepackage{tikz}
\usetikzlibrary{calc,trees,positioning,arrows,chains,shapes.geometric,%
    decorations.pathreplacing,decorations.pathmorphing,shapes,%
    matrix,shapes.symbols}

\tikzset{
>=stealth',
  punktchain/.style={
    rectangle,
    rounded corners,
    draw=black, thick,
    minimum height=3em,
    text centered,
    on chain},
  line/.style={draw, thick, <-},
  element/.style={
    tape,
    top color=white,
    bottom color=blue!50!black!60!,
    minimum width=8em,
    draw=blue!40!black!90, very thick,
    text width=10em,
    minimum height=3.5em,
    text centered,
    on chain},
  every join/.style={->, thick,shorten >=1pt},
  decoration={brace},
  tuborg/.style={decorate},
  tubnode/.style={midway, right=2pt},
}

\title[Smoothness for Fano orbifolds and Lagrangian bases]{Smoothness criterion for Fano orbifolds \\ and bases of Lagrangian fibrations}

\subjclass[2020]{14J42 (Primary); 14J32, 14J27 (Secondary)}
\keywords{Quotient singularities, Canonical stacks, Hyper-K\"ahler manifolds, Lagrangian fibrations}

\author[Zhiyu Liu]{Zhiyu Liu}
\address{\parbox{0.9\textwidth}{Department of Mathematics, Princeton University\\[1pt] Princeton, NJ 08544, USA
\vspace{1mm}}}
\email{zl3301@princeton.edu}
\urladdr{\url{https://sites.google.com/view/zhiyuliu}}

\begin{document}
\begin{abstract}
We establish a smoothness criterion for Fano varieties with quotient singularities. As an application, we show that if $X\to B$ is a Lagrangian fibration from a compact hyper-K\"ahler manifold and $B$ has quotient singularities, then $B$ is a projective space.
\end{abstract}
\maketitle

\section{Introduction}\label{sec:introduction}

Let $X$ be a normal projective variety with quotient singularities. By \cite[Remark~4.9]{FMN}, there exists a canonical smooth Deligne--Mumford stack $\cX\to X$ whose coarse moduli space is $X$. We denote by $I\cX\coloneqq \cX\times_{\cX\times \cX} \cX$ the inertia stack of $\cX$. The natural morphism $\cX\to I\cX$ is an open and closed embedding, and its image $I\cX^e$ is called the identity component.

Our main result is the following criterion.

\begin{theorem}\label{thm:smooth}
Let $X$ be a Fano variety with quotient singularities. Assume that
\begin{enumerate}[label=(\roman*)]
    \item there exists a surjective morphism $A\to X$ from an abelian variety $A$, and

    \item for any connected component $\cI\subset I\cX$ with $\cI\neq I\cX^e$, we have $$c_{\dim(\cI)}(T_{\cI})=0\in \CH^*(\cI)_{\QQ}.$$
\end{enumerate}
Then $X$ is a product of projective spaces.
\end{theorem}

By \cite{DHP} and simple connectedness of Fano manifolds, it suffices to prove that $X$ is smooth. To show this, we may assume that $A\to X$ is finite and consider the lift $f\colon A\to \cX$. We denote by $\D\colon A\times_{\cX} A\to A$ the difference of two projections of $A\times_{\cX} A$. Therefore, the fiber $\D^{-1}(0)=A\times_{\cX}I\cX$ contains the identity part $A\times_{\cX}I\cX^e$, and its complement  $\D^{-1}(0)\setminus A\times_{\cX}I\cX^e$ encodes the nontrivial stabilizers of $\cX$.

It suffices to show $\D^{-1}(0)= A\times_{\cX}I\cX^e$. To this end, motivated by the proof of \cite[Theorem~1.3]{xu-muller} (see Remark~\ref{rmk:muller-xu}), we use the excess intersection formula to express the refined intersection with $\D^{-1}(0)$ in terms of the top Chern classes of the inertia components. By the positivity of intersection multiplicities and Theorem~\ref{thm:smooth}(ii), we show in Lemma~\ref{lem:proper-differences} that $\D(R_0)\ne A$ for every irreducible component $R_0\subset A\times_{\cX} A$ meeting $\D^{-1}(0)\setminus A\times_{\cX}I\cX^e$.

Another constraint on such a component $R_0$ comes from the action of abelian varieties. In Lemma~\ref{lem:polarized-differences}, we show that for such a component $R_0$, $\D(R_0)$ is an abelian subvariety of $A$ and its orthogonal complement $\D(R_0)^\perp$ acts on the normalization of $R_0$. Moreover, as $X$ is Fano, $f$ has an ample ramification divisor. These allow us to construct a second component $R'\subset A\times_{\cX} A$ in Lemma~\ref{lem:ramification-relation} with $\D(R_0)\neq \D(R')$. By composing arrows parameterized by $R_0$ and $R'$, at the end of Section~\ref{sec:criterion}, we construct a third component $\mu(T)$ of $A\times_{\cX} A$ meeting $\D^{-1}(0)\setminus A\times_{\cX}I\cX^e$ with  $\dim(\D(\mu(T)))>\dim(\D(R_0))$. As there are only finitely many such $R_0$, we get a contradiction.

\subsection{Bases of Lagrangian fibrations}

Our main application of Theorem~\ref{thm:smooth} concerns the bases of Lagrangian fibrations.

Let $X\to B$ be a Lagrangian fibration from a compact hyper-K\"ahler manifold. In Proposition~\ref{prop:finite-cover}, we deform $X\to B$ to another Lagrangian fibration $X'\to B$ with a second fibration $X'\to B'$. Therefore, a general fiber of $X'\to B'$ gives a cover of $B$ by an abelian variety and implies assumption (i) in Theorem~\ref{thm:smooth}. Note that similar arguments for $\mathrm{K3^{[n]}}$ and OG10-types are contained in \cite{shah}. 

In Section~\ref{sec:smooth}, we further verify assumption (ii) in Theorem~\ref{thm:smooth} for $B$, hence obtain:

\begin{theorem}\label{thm:main}
Let $X\to B$ be a Lagrangian fibration from a compact hyper-K\"ahler manifold of dimension $2n$. If $B$ has quotient singularities, then $B\cong\PP^n$.
\end{theorem}

When $\dim X=4$, Theorem~\ref{thm:main} is proved in \cite{Huy-Xu,Ou2019} without assuming quotient singularities. When $B$ is smooth, this is the theorem of Hwang \cite{Hwang2008}; see \cite{bakker-schnell,Li-Tosatti} for alternative proofs. In both cases, the method of Theorem~\ref{thm:main} is different from the existing ones.

\subsection*{Acknowledgments}
I would like to thank Chenyang Xu and Yuchen Liu for useful discussions.

\subsection*{AI disclosure}
The author's investigation of Theorem~\ref{thm:main} began with \cite[Conjecture~1.4]{liu-liu-xu}. Following the appearance of \cite{xu-muller}, which treats the case of isolated quotient singularities germs, the author realized that the canonical stack of the base could be used to treat quotient singularities uniformly and derive topological obstructions from the global hyper-K\"ahler geometry. The author's initial approach to Theorem~\ref{thm:main} was to pull back the Lagrangian fibration to a non-identity component $\cI$ of the inertia stack and use a stacky version of the group action constructed in \cite{dCKS} to construct a non-compact Lagrangian fibration over $\cI$. The compactification techniques of \cite{sacca,liu-liu-xu} would then produce a Lagrangian fibration over $\cI$. This would impose Fano-type topological and intersection-theoretic constraints on $\cI$, while the vanishing $c_{\dim \cI}(T_{\cI})=0$ can be regarded as a Calabi--Yau property, hence ultimately ruling out the existence of $\cI$. Although this approach could be carried out, the resulting proof was approximately 50 pages.

While the author was pursuing this approach, OpenAI released a collection of mathematical manuscripts generated by an internal model on October~6, 2026. The manuscript \cite{OAI}, which is approximately 100 pages long, announces a proof of Theorem~\ref{thm:main} without assuming quotient singularities. The author has found
the \cite{OAI}’s exposition difficult to follow and has not yet been able to verify many of its details or assess the correctness of its arguments.

However, on a preliminary reading, its argument consists of two parts: the first establishes that the base has quotient singularities, while the second proves smoothness under this assumption through a detailed study of certain vector fields on stacks and local analytic arguments. After reading the manuscript with the assistance of an LLM, the author found a substantially simpler proof of the second part (i.e.~Theorem~\ref{thm:main}) by combining ideas from \cite{OAI} with the author's earlier approach. More precisely, the author obtained a simple intersection-theoretic proof of the more general Theorem~\ref{thm:smooth} and substantially simplified the proof of Proposition~\ref{prop:finite-cover}. The current paper is written by the author, who takes full responsibility for the paper.

\section{Smoothness criterion}\label{sec:criterion}

Let $X$ be a normal projective variety of dimension $n$ with $-K_X$ $\QQ$-Cartier and ample, i.e.~$X$ is a Fano variety. Throughout this section, we assume that $X$ has quotient singularities and there is a surjection $\pi\colon A\to X$ from an abelian variety $A$. By \cite[Corollary~2.2]{DHP}, after replacing $A$ by its quotient, we may also assume that $\pi$ is finite.

Let $p\colon \cX\to X$ be the canonical stack, which is an isomorphism over $X_{\reg}$ by \cite{FMN}. By the universal property of the canonical stack \cite[Theorem~4.6]{FMN}, we have a lift
\[f\colon A\to\cX\]
with $p\circ f=\pi$. It is clear that $f$ is representable, flat, and finite.

We set $R\coloneqq A\times_{\cX}A$, with projections $s,t\colon R\to A$. We define $\D\coloneqq t-s\colon R\to A$. For the normalization $\nu\colon Z\to R$ of an irreducible component of $R$, we write $\D_Z=\D\circ \nu$ and use $s,t$ also for the induced morphisms on $Z$.

Since both projections $s,t$ of $R$ are finite flat and generically \'etale, we see that $R$ is reduced and pure of dimension $n$, and every component dominates $A$ under both projections. Moreover,
\[\D^{-1}(0)=A\times_{\cX}I\cX.\]
We denote by $i\colon A\to A\times_{\cX}I\cX^e$ the natural inclusion. Hence, $i(A)=A\times_{\cX}I\cX^e$ is the identity part of $\D^{-1}(0)$, which is open and closed.

\begin{lemma}\label{lem:proper-differences}
Assume that $c_{\dim(\cI)}(T_{\cI})=0\in \CH^*(\cI)_{\QQ}$ for any connected component $\cI\subset I\cX$ with $\cI\neq I\cX^e$. If an irreducible component $R_0\subset R$ satisfies $R_0\cap (\D^{-1}(0)\setminus i(A))\neq \emptyset$, then $\D(R_0)\ne A$.
\end{lemma}

\begin{proof}
Let $j\colon\{0\}\hookrightarrow A$ and let
\[q_I\colon\D^{-1}(0)=A\times_{\cX}I\cX \lra I\cX\]
be the projection, which is finite flat as a base change
of $f$. The diagonal morphisms $\Delta_A$ and
$\Delta_{\cX}$ are representable regular local embeddings. Therefore, the flat base change and commutativity of refined Gysin maps \cite[Section~6.2 and Theorem~6.4]{Fulton1998},
\cite[Sections~3.1 and~3.3]{Kresch1999} give
\[\begin{aligned}
j^!([R])
&=(\Delta_A^!\circ\Delta_{\cX}^!)([A\times A])
=(\Delta_{\cX}^!\circ\Delta_A^!)([A\times A])\\
&=\Delta_{\cX}^!([A])
=q_I^*\bigl(\Delta_{\cX}^!([\cX])\bigr).
\end{aligned}\]
Note that $\Delta_A$ is the pullback of $j$ under the difference morphism $A\times A\to A$, $(a,b)\mapsto b-a$.

Let $\cI\subset I\cX\setminus I\cX^e$ be a connected component, with evaluation $e\colon\cI\to\cX$ and tautological automorphism $g$. Thus, we have a decomposition
\[e^*(T_{\cX})\cong T_{\cI}\oplus N_e\]
where $N_e$ is the normal bundle of $e$. Thus, the excess intersection formula
\cite[Section~6.3]{Fulton1998} and our assumption give
\begin{equation}\label{eq:zero-gysin}
\left.j^!([R])\right|_{A\times_{\cX}\cI}
=c_{\dim(\cI)}(q_I^*(T_{\cI}))\cap[A\times_{\cX}\cI]=0.
\end{equation}
Now, we choose a connected component $W_0$ of
$\D^{-1}(0)\setminus i(A)$ meeting $R_0$.
It lies in $A\times_{\cX}\cI$ for some component $\cI\subset I\cX\setminus I\cX^e$, so
\eqref{eq:zero-gysin} gives $\left.j^!([R])\right|_{W_0}=0$.

Let
\[R\xra{\,h\,}S\xra{\,\ell\,}A \]
be the Stein factorization of $\D\colon R\to A$. Hence, $h$ has connected geometric fibers and $\ell$ is finite. Moreover, $W_0$ is the inverse image of a $0$-dimensional component of $\ell^{-1}(0)$ supported at some point $\xi\in S$.

Suppose for a contradiction that $\D(R_0)=A$. Then $h(R_0)$ is an $n$-dimensional component of $S$ containing $\xi$, with positive coefficient in the effective cycle $h_*([R])$. Write $h_*([R])=\sum_i c_i[V_i]$, where $c_i>0$ and $\dim(V_i)=n$. Since $\ell$ is finite, $V_i\times_A\{0\}$ is zero-dimensional. As $j$ is a regular embedding of codimension $n$, \cite[Proposition~7.1(a)]{Fulton1998} gives
\[0=\operatorname{coeff}_{[\xi]}\bigl((j^!\circ h_*)([R])\bigr)=\sum_{\xi\in V_i}c_i\operatorname{coeff}_{[\xi]}\bigl(j^!([V_i])\bigr)>0,\]
which is a contradiction.
\end{proof}

\begin{remark}\label{rmk:muller-xu}
If $G$ is a finite group acting on $A$ and $f\colon A\to [A/G]=\cX$ is the quotient morphism, then we have a more explicit description. In this case, $R= \coprod_{g\in G} A\times \{g\}$ and $\D\colon R\to A$ is given by $(a,g)\mapsto a-g\cdot a$. Therefore, we have $\D^{-1}(0)=\coprod_{g\in G} A^g$ and $i(A)=A^e$, where $A^g$ is the fixed locus of the action on $A$ by $g$ and $e\in G$ is the identity. Hence, in the setting of Lemma~\ref{lem:proper-differences}, we have $R_0=A\times \{g_0\}$ for some $g_0\in G$ with $g_0\neq e$ and $A^{g_0}\neq \varnothing$. Since $R_0$ is an abelian variety, the condition $\D(R_0)= A$ is then equivalent to saying that $R_0\to A$ is finite, i.e.~$A^{g_0}$ consists of isolated points. This reduces to the setting of \cite[Theorem~1.3]{xu-muller}.
\end{remark}

Fix an ample line bundle $H$ on $X$ and set $L\coloneqq \pi^*H$. Thus, $s^*L\cong t^*L$. Let $\lambda_L\colon A\to\widehat A$ be the homomorphism given by $a\mapsto t_a^*L\otimes L^{-1}$. For an abelian subvariety $D\subset A$, we set
\[D^\perp\coloneqq \ker\bigl(A\xra{\lambda_L}\widehat A\lra\widehat D\bigr)^0,\]
where $(-)^0$ denotes the identity component. In particular, we have an isogeny $D\times D^\perp\to A$ and a diagram 
\[\begin{tikzcd}
	& {D^{\perp}} \\
	D & A \\
	{\widehat{D}} & {\widehat{A}}
	\arrow[hook, from=1-2, to=2-2]
	\arrow[hook, from=2-1, to=2-2]
	\arrow["{\lambda_L}", two heads, from=2-2, to=3-2]
	\arrow[two heads, from=3-1, to=2-1]
	\arrow[two heads, from=3-2, to=3-1]
\end{tikzcd}\]
Moreover, by Poincar\'e complete reducibility, we have
\[\dim(D^\perp)=n-\dim(D)\]
and $(D^\perp)^\perp=D$.

\begin{lemma}\label{lem:polarized-differences}
Let $Z$ be the normalization of an irreducible component of $R$ with $0\in\D_Z(Z)\subsetneq A$. 

\begin{enumerate}
    \item $\D_Z(Z)$ is an abelian subvariety of $A$.

    \item There is a $\D_Z(Z)^{\perp}$-action on $Z$ with $s(kz)=k+s(z)$, $t(kz)=k+t(z)$, and $\D_Z(kz)=\D_Z(z)$ for any $k\in \D_Z(Z)^{\perp}$ and $z\in Z$.
\end{enumerate}
\end{lemma}

\begin{proof}
We set $r=\dim(\D_Z(Z))$. For any point $a\in\D_Z(Z)$, let $Z_a$ be an irreducible component of $\D_Z^{-1}(a)$. Then we have $\dim(Z_a)\geq n-r$, $t|_{Z_a}=t_a\circ s|_{Z_a}$, and $(s|_{Z_a})^*(\lambda_L(a))\cong\cO_{Z_a}$.

Let $K_a\subset A$ be the abelian subvariety generated by differences of points of $s(Z_a)$. Since $s$ is finite,
we have $\dim(K_a)\ge n-r$. The Albanese homomorphism of a resolution of $Z_a$ maps onto $K_a$, so its dual has finite kernel. Consequently $\lambda_L(a)|_{K_a}$ is torsion.

For an abelian subvariety $K\subset A$ with $\dim(K)\ge n-r$, consider the homomorphism
\[\psi_K\colon A\lra\widehat K,\qquad a\longmapsto\lambda_L(a)|_K.\]
By definition, $(\ker(\psi_K))^0=K^\perp$. For every positive integer $m$, the finiteness of $\widehat K[m]$ gives
\[\bigl(\ker([m]\circ\psi_K)\bigr)^0=K^\perp.\]
Thus, $\ker([m]\circ\psi_K)$ is a union of translations $b+K^\perp$ for finitely many $b\in A$.

For any point $a\in\D_Z(Z)$, the preceding argument gives an abelian subvariety $K_a$ with $\dim(K_a)\ge n-r$ such that $\lambda_L(a)|_{K_a}$ is torsion. Hence, $a\in\ker([m]\circ\psi_{K_a})$ for some positive integer $m$. Since $A$ has only countably many abelian subvarieties, it follows that $\D_Z(Z)$ is covered by $b+\bigl(\ker([m]\circ\psi_{K_a})\bigr)^0=b+K_a^{\perp}$ for countably many $a,b$, and $m$. Now, choose a very general point $a\in\D_Z(Z)$ outside every proper intersection with these countably many $b+K_a^{\perp}$. Hence, we have $a\in b+K_a^\perp$ and $\dim(K_a^\perp)\le r$. By our choice of $a$, we also have $\D_Z(Z)\subset b+K_a^\perp$. Since $b+K_a^{\perp}$ is irreducible, $r=\dim(\D_Z(Z))$, and $0\in\D_Z(Z)$, we obtain
\[\D_Z(Z)=b+K_a^\perp=K_a^\perp.\]
This proves (a). Moreover, for every irreducible component of $\D_Z^{-1}(a)$ for a very general point $a\in \D_Z(Z)$, we have $K_a=\D_Z(Z)^\perp$.

To prove (b), we choose a very general point $a\in\D_Z(Z)$. Let $Z_a$ be any irreducible component of the reduced fiber $\D_Z^{-1}(a)$, and choose $z_a\in Z_a$. By the preceding argument, $K_a=\D_Z(Z)^\perp$, so $s(Z_a)\subset s(z_a)+\D_Z(Z)^\perp$. Since $s$ is finite, we have
\[n-r\le\dim(Z_a)=\dim(s(Z_a))\le\dim(\D_Z(Z)^\perp)=n-r.\]
Thus $s(Z_a)=s(z_a)+\D_Z(Z)^\perp$. Since $t|_{Z_a}=t_a\circ s|_{Z_a}$, we obtain
\[(s,t)(Z_a)=\{(s(z_a)+k,t(z_a)+k)\mid k\in\D_Z(Z)^\perp\}.\]
As $a$ varies over these very general points $a\in\D_Z(Z)$, the corresponding fibers have dense union in $Z$. Therefore, taking closures shows that $\Gamma\coloneqq (s,t)(Z)\subset A\times A$ is preserved by diagonal $\D_Z(Z)^\perp$-translation.

The morphism $(s,t)\colon Z\to\Gamma$ is finite because $s$ is finite. Over $X_{\reg}$, the morphism $(s,t)\colon R\to A\times A$ is a closed immersion. Hence $(s,t)\colon Z\to\Gamma$ is birational, and $Z$ is the normalization of $\Gamma$. Since $\D_Z(Z)^\perp$ is smooth and $Z$ is normal, $\D_Z(Z)^\perp\times Z$ is normal. The universal property of normalization \cite[Tag~035Q]{StacksProject} then gives a unique lift of the morphism
\[\D_Z(Z)^\perp\times Z\lra\Gamma,\qquad (k,z)\longmapsto(k+s(z),k+t(z))\]
to $Z$. Write this lift as $(k,z)\mapsto kz$. Hence, the uniqueness shows that this is a $\D_Z(Z)^\perp$-action on $Z$. By construction, we see that $s(kz)=k+s(z)$, $t(kz)=k+t(z)$, and $\D_Z(kz)=\D_Z(z)$, which proves (b).
\end{proof}

\begin{lemma}\label{lem:ramification-relation}
Let $D\subsetneq A$ be an abelian subvariety. Then there are a prime ramification divisor $Q\subset A$ of $f$ and an irreducible component $R'\ne i(A)$ of $R$ such that

\begin{enumerate}
    \item we have a surjection $Q\twoheadrightarrow A/D^\perp$, and 

    \item $i(Q)\subset R'$ and $\D(R')\not\subset D$.
\end{enumerate}
\end{lemma}

\begin{proof}
Let $\Ram(f)\subset A$ be the ramification divisor of $f$. Since $\omega_A\cong\cO_A$ and $\cX$ has no stabilizer in codimension $1$, for sufficiently divisible $m>0$, we have
\[\cO_A(m\Ram(f))\cong\pi^*\cO_X(-mK_X).\]
In particular, it is ample. Since $\dim(D^\perp)>0$, every $D^\perp$-coset meets $\Ram(f)$. Hence, some irreducible component $Q$ of $\Ram(f)$ maps onto $A/D^\perp$. This proves (a).

For part (b), note that the generic point of $Q$ lies over $X_{\reg}$. If $\Delta_{A/\cX}(A)=i(A)$ is the only irreducible component of $R$ through the generic point of $\Delta_{A/\cX}(Q)=i(Q)$, we see that $\Delta_{A/\cX}$ is an open immersion near the generic point of $i(Q)$. Therefore, \cite[Tag~02GE]{StacksProject} implies that $f$ is unramified at the generic point of $Q$, a contradiction. Thus there is an irreducible component $R'\ne i(A)$ with $i(Q)\subset R'$.

Now, suppose for a contradiction that $\D(R')\subset D$, and let $\nu\colon Z\to R'$ be the normalization. By Lemma~\ref{lem:polarized-differences}, $\D(R')$ is an abelian subvariety and $D^\perp\subset\D(R')^\perp$ acts on $Z$ by diagonal translation. Since $\nu^{-1}(i(A))$ is open and closed in $\D_Z^{-1}(0)$ and $D^\perp$ is connected, it is $D^\perp$-invariant. Thus $s(\nu^{-1}(i(A)))$ is a $D^\perp$-invariant closed proper subset of $A$ containing $Q$. Thus, we obtain
\[A=D^\perp+Q\subset s(\nu^{-1}(i(A)))\subsetneq A,\]
a contradiction.
\end{proof}

\begin{proof}[Proof of Theorem~\ref{thm:smooth}]
Suppose that $\D^{-1}(0)\ne i(A)$. We choose the normalization $\nu\colon Z\to R$ of an irreducible component meeting $\D^{-1}(0)\setminus i(A)$ for which $\dim(\D_Z(Z))$ is maximal. By Lemmas~\ref{lem:proper-differences} and~\ref{lem:polarized-differences}, $\D_Z(Z)\subset A$ is a proper abelian subvariety, and $\D_Z(Z)^\perp$ acts on $Z$. Since $\D_Z(Z)^\perp$ is connected, its action preserves $\nu^{-1}(\D^{-1}(0)\setminus i(A))$.

We fix $Q,R'$ as in Lemma~\ref{lem:ramification-relation}, applied to $\D_Z(Z)$. Since $Q+\D_Z(Z)^\perp=A$, the orbit of any point in $\nu^{-1}(\D^{-1}(0)\setminus i(A))$ meets $s^{-1}(Q)$. Let $E$ be an irreducible component of $s^{-1}(Q)$ meeting this set. Since $Q$ is Cartier and $s$ is finite, we have $\dim(E)=n-1$ and $s(E)=Q$. As the image $\D_Z(Z)^\perp\cdot E$ is closed and irreducible, and its image under $s$ is $A$, we get $\D_Z(Z)^\perp\cdot E=Z$ and $\D_Z(E)=\D_Z(Z)$.

Let $T$ be an irreducible component of $R'\times_{t,A,s}Z$ containing the image of
\[E\lra R'\times_{t,A,s}Z,\qquad z\longmapsto(i(s(z)),z).\]
Note that $R'\times_{t,A,s}Z$ is the fiber of the morphism $R'\times Z\to A$ given by $(r',z)\mapsto t(r')-s(z)$ over $0$. Since $\dim(R'\times Z)=2n$ and $\dim(A)=n$, we get $\dim(T)\ge n$. Both projections are finite, so $\dim(T)=n$ and $T$ maps onto both $R'$ and $Z$.

We define a morphism $\mu\colon T\to R$ as the composition
\[T\xra{(\pr_1|_T,\nu\circ \pr_2|_T)}R\times_{t,A,s}R\to R,\]
where $\pr_1\colon R'\times_{t,A,s}Z\to R'$ and $\pr_2\colon R'\times_{t,A,s}Z\to Z$ are projections, and $R\times_{t,A,s}R\to R$ is the projection to the first and third factors, under the canonical identification $$R\times_{t,A,s}R= A\times_{\cX}A\times_{\cX}A.$$ Therefore, we have 
\begin{equation}\label{eq1}
s\circ \mu=s\circ \pr_1|_T
\end{equation}
and
\begin{equation}\label{eq2}
\D\circ \mu=\D\circ \pr_1|_T+\D_Z\circ \pr_2|_T.
\end{equation}
By \eqref{eq1}, we know that $\mu$ is a morphism between schemes finite over $A$, so $\mu$ is finite. Hence, $\mu(T)$ is an irreducible component of $R$. For $z\in E$, composition with the identity gives $\mu(i(s(z)),z)=\nu(z)$. Since $E$ meets $\nu^{-1}(\D^{-1}(0)\setminus i(A))$ and $\D_Z(E)=\D_Z(Z)$, we obtain $$\mu(T)\cap(\D^{-1}(0)\setminus i(A))\ne\emptyset$$ and $\D_Z(Z)\subset\D(\mu(T))$. Now, by \eqref{eq2}, the surjectivity of $T\to R'$, and $\D(R')\not\subset\D_Z(Z)$, we have $\D_Z(Z)\subsetneq\D(\mu(T))$. As both images are closed and irreducible, $\dim(\D(\mu(T)))>\dim(\D_Z(Z))$, contradicting maximality.

Thus, $\D^{-1}(0)=i(A)$ and faithfully flat descent along $f\colon A\to\cX$ yields $I\cX=I\cX^e$. Therefore, all stabilizers of $\cX$ are trivial, so its coarse moduli morphism $p$ is an isomorphism. Since $\cX$ is smooth, $X$ is smooth. Now, the result follows from \cite[Theorem~1.1]{DHP} and simple connectedness of $X$.
\end{proof}

\section{Finite covers by abelian varieties}\label{sec:fixed-base}

We say a complex manifold $X$ is a compact hyper-K\"ahler manifold if $X$ is a simply connected compact K\"ahler manifold with $\rH^{2,0}(X)$ spanned by a single holomorphic symplectic form. In this section, we construct finite cover of bases of Lagrangian fibrations on $X$ by abelian varieties.

\begin{proposition}\label{prop:finite-cover}
Let $\pi\colon X\to B$ be a Lagrangian fibration of a compact hyper-K\"ahler manifold $X$ of dimension $2n$. Then there exists a finite surjective morphism $A\twoheadrightarrow B$ from an abelian variety $A$.
\end{proposition}

\begin{remark}
Combining Proposition~\ref{prop:finite-cover} with \cite{DHP}, this gives a simple proof of \cite[Theorem~1.2]{Hwang2008}.
\end{remark}

Note that when $b_2(X)\geq 5$, we can use the argument in \cite[Proposition~5.3]{shah} to get a simple proof of Proposition~\ref{prop:finite-cover}. The following argument, which also covers $b_2(X)=4$, uses a similar idea and is more complicated.

In the rest of this section, we fix a very ample line bundle $L$ on $B$ and set $e\coloneqq c_1(\pi^*L)\in \rH^2(X,\ZZ)$. We denote the lattice $\rH^2(X,\ZZ)$ with the Bogomolov--Beauville--Fujiki form $q$ by $(\Lambda,q)$. Then the period domain is
\[\Omega_\Lambda\coloneqq \{[\sigma]\in\PP(\Lambda_\CC)\mid q(\sigma)=0,\ q(\sigma,\bar\sigma)>0\}.\]
We denote by $\cK_X$ the K\"ahler cone of $X$.

Let $\cP\colon\Def(X)\to\Omega_\Lambda$ be the local period map. We set $$\Def(X,e)\coloneqq \cP^{-1}(\Omega_\Lambda\cap\PP((e^\perp)_\CC)).$$
By \cite[Theorem~1.1]{MatsushitaDeformations}, this is the local deformation space of the pair $(X,\pi^*L)$.

\begin{lemma}[{\cite[Lemma 3.7]{syz-moduli-k3}}]\label{lem:fixed-base}
There exists an open subset $D\subset\Def(X,e)$ with the Kuranishi family $\mathscr X\to D$ and a morphism $\tpi\colon \mathscr X\to \mathscr B$ over $D$, such that $\tpi_t$ is a Lagrangian fibration over $B$ for each $t\in D$ and $\tpi_0=\pi$ for a point $0\in D$.
\end{lemma}

By Lemma~\ref{lem:fixed-base} and the local Torelli theorem (cf.~\cite[Theorem~4.1]{Huybrechts2012}), we may assume that $X$ is projective. Let $h\in \rH^2(X,\ZZ)$ be an ample class. Then $q(e,h)>0$. We may also assume that $$\mathrm{P}_X\coloneqq \langle \Re(\sigma), \Im(\sigma) \rangle_{\RR}\subset \Lambda_{\RR}$$ is rational, where $\sigma$ is a holomorphic 2-form on $X$. Note that $\mathrm{P}_X$ is independent of the choice of $\sigma$. Let $\mathfrak M$ be the connected component of the marked moduli space that contains $(X,\id)$.

For a field $\FF$ and an $\FF $-vector space $V$, we denote by $\langle S\rangle_{\FF}$ the $\FF$-subspace of $V$ spanned by a subset $S\subset V$. Set $U\coloneqq (\mathrm{P}_X\cap\Lambda_{\QQ})\oplus\langle e,h\rangle_{\QQ}$ and
\[\HH(U)\coloneqq \{[v]\in \PP(U_{\RR})\mid q(v)<0\}.\]
In particular, $q|_U$ has signature $(3,1)$. We use the orientation of $\mathrm{P}_X\oplus\RR h$ determined by $(\Re(\sigma),\Im(\sigma),h)$. Therefore, it induces an orientation for any positive 3-plane of $\Lambda_{\RR}$.

A positive 3-plane is \emph{realizable} if it is $\langle \mathrm{P}_Y,\kappa \rangle_{\RR}$ for a holomorphic 2-form $\sigma$ and a K\"ahler class $\kappa$ on $Y$, where $Y$ is in the fixed component $\mathfrak M$. By \cite[Theorem~4.9]{AV15}, these are precisely the
3-planes orthogonal to no monodromy birationally minimal (MBM) class, and they form an open set in the space of 3-planes.

For any $[v]\in \PP(U_{\RR})$, we define a 3-plane $W_{[v]}\coloneqq v^\perp\cap U_\RR$. Let $\HH(U)_{\mathrm{nr}}\subset\HH(U)$ be the locus consisting of $[v]$ with $W_{[v]}$ non-realizable. Hence, it is a closed subset. Since $\mathrm{P}_X\oplus\RR h$ is realizable, no MBM class lies in $U^\perp$. For $[v]\in \HH(U)_{\mathrm{nr}}$, an orthogonal MBM class $\delta$ therefore satisfies
\[0\ne\pr_U(\delta)\in U_\QQ\cap\RR v,\]
where $\pr_U\colon \Lambda_{\RR}\to U_{\RR}$ is the orthogonal projection. Hence
\begin{equation}\label{eq:rational-obstructions}
 \HH(U)_{\mathrm{nr}}\subset\HH(U)\cap \PP(U_{\QQ}).
\end{equation}
Let $\Mon^2(X)\subset O^+(\Lambda,q)$ be the degree-two monodromy group \cite[Definition~1.1]{Markman2011}. Thus, every element in $\Mon^2(X)$ that preserves $U$ also preserves $\HH(U)_{\mathrm{nr}}$.

\subsection{A parallel transport}
Let $U^-$ be the connected component of 
\[\{v\in U_{\RR}\mid q(v)<0\}\]
whose boundary contains $\RR_{>0}e$. We set
\[\partial\HH(U)\coloneqq \{[v]\in\PP(U_\RR)\mid q(v)=0\}\]
and
\[\partial U^-\coloneqq \{v\in \overline{U^-}\setminus \{0\}\mid q(v)=0\}.\]
Define $\ell(v,w)\coloneqq\{[sv+tw]\mid s,t>0\}\subset \HH(U)$ for any non-proportional $v,w\in \partial U^-$. We set
\[w_*\coloneqq e-\frac{2q(e,h)}{q(h)}h\in \partial U^-,\]
hence $[w_*]\in \partial \HH(U)$.

\begin{lemma}\label{lem:irrational-axis}
For every neighborhood $V$ of $[w_*]$ in $\partial \HH(U) $, there exists $\gamma\in\Mon^2(X)$ with
\begin{itemize}
    \item $\gamma(U)=U$ and $\gamma|_{U^\perp}=\id$,

    \item there exists $\rho>1$ such that $\gamma (v_\pm)=\rho^{\pm1}v_\pm$ with $v_{\pm}\in \partial U^-$ and $[v_\pm]\in V$, and

    \item $\langle v_-,v_+\rangle_\RR\cap U_\QQ=0$.
\end{itemize}
\end{lemma}

\begin{proof}
We choose $b\in\QQ w_*$ with $q(e,b)=1$ and nonzero vectors $x,y\in\langle e,b\rangle^\perp_{\QQ}\cap U_\QQ$ with $q(x,y)=0$. For any 
$a\in\{e,b\}$ and $z\in\langle e,b\rangle^\perp_{\QQ}\cap U_\QQ$, we define an operator $T_{a,z}\colon \Lambda_{\QQ}\to \Lambda_{\QQ}$ by
\begin{equation}\label{eq:null-rotation}
 T_{a,z}(v)\coloneqq v+q(v,a)z-q(v,z)a-\frac{q(z)}2q(v,a)a.
\end{equation}
Thus, we have $T_{a,z}\in O^+(\Lambda_\QQ,q)$, $T_{a,z}|_{U^\perp}=\id$, and $T_{a,z}^N=T_{a,Nz}$. By \cite[Lemma~7.5]{Markman2011}, we may rescale $x,y$ so that
\[T_{e,x},T_{b,y}\in\Mon^2(X)\]
and $k\coloneqq\frac{q(x)q(y)}4\ge25$.

Write $\gamma_0=T_{e,x}\circ T_{b,y}$. Then the characteristic polynomial $\chi_{\gamma_0|_U}(t)$ of $\gamma_0|_U\colon U\to U$ is
\[\chi_{\gamma_0|_U}(t)=t^4-(4+k)t^3+(6-2k)t^2-(4+k)t+1.\]
Moreover, $k=\tr(\gamma_0|_U)-4\in\ZZ$. Set
\[z_\pm=\frac{k+4\pm\sqrt{k^2+16k}}2.\]
Since $k\ge25$, we have $(k+7)^2<k^2+16k<(k+8)^2$, so $z_+\notin\QQ$. Moreover, $z_+>2$ and $-2<z_-<2$, and
\[\chi_{\gamma_0|_U}(t)=(t^2-z_+t+1)(t^2-z_-t+1).\]
Hence the real eigenvalues of $\gamma_0|_U$ are $\rho,\rho^{-1}$, where $\rho=(z_++\sqrt{z_+^2-4})/2>1$. Their eigenvectors $v_+$ and $v_-$ satisfy $q(v_\pm)=q(\gamma_0(v_\pm))=\rho^{\pm2}q(v_\pm)$. Thus $q(v_\pm)=0$, and we may take $v_\pm\in\partial U^-$.

Note that for $v\in\langle v_-,v_+\rangle_\RR\cap U_\QQ$, we have $(\gamma_0^2+\id)(v)=z_+\gamma_0(v)$. Since both $\gamma_0^2(v)+v$ and $\gamma_0(v)$ lie in $U_\QQ$ and $z_+\notin\QQ$, we obtain $\gamma_0(v)=0$. Hence, $\langle v_-,v_+\rangle_\RR\cap U_\QQ=0$.

Note that for every nonzero vector $v\in U_{\RR}\setminus \RR b$ with $q(v)=0$, \eqref{eq:null-rotation} gives
\[\lim_{N\to \infty}[T_{b,y}^N(v)]=[b]=[w_*].\]
Thus, we may fix a sufficiently large $N$ and set
\[\gamma\coloneqq T_{b,y}^N\circ\gamma_0\circ T_{b,y}^{-N}\in\Mon^2(X).\]
After replacing $v_\pm$ by eigenvectors $T_{b,y}^N(v_\pm)$ of $\gamma|_U$, we get the result.
\end{proof}

In particular, \eqref{eq:rational-obstructions} gives
\begin{equation}\label{eq:three-geodesics}
\HH(U)_{\mathrm{nr}}\cap\bigl(\ell(v_-,v_+)\cup\ell(e,v_+)\cup\ell(e,v_-)\bigr)=\emptyset.
\end{equation}

\subsection{An additional fibration}

For $[w]$ near $[w_*]$ in $\partial \HH(U)$, we take the representative $w\in\partial U^-$ with $q(e,w)=-1$, and set $\alpha(w)\coloneqq -q(h,w)$, $h_w\coloneqq \alpha(w)e-q(e,h) w$, and
\[\mathrm{P}(e,w)\coloneqq \langle e,w\rangle^\perp_{\RR}\cap U_\RR.\]
At $[w]=[w_*]$, we have $\mathrm{P}(e,w)=\mathrm{P}_X$ and $h_w=h$. Note that $\mathrm{P}(e,w)$ has an orientation induced by $\mathrm{P}_X$. Therefore, Local Torelli \cite[Theorem~4.1]{Huybrechts2012} and Lemma~\ref{lem:fixed-base} give a Lagrangian fibration $\pi_w\colon X_w\to B$ with $e\in\overline{\cK}_{X_w}$ and $h_w\in\cK_{X_w}$. For $t\ge0$, the 3-plane
\[W_{[(\alpha(w)+t)e+q(e,h) w]}=\mathrm{P}(e,w)\oplus\RR(h_w+te)\]
is therefore realizable.

Now, we shrink to a neighborhood $V$ of $[w_*]$ and choose $M>0$ such that $0<\alpha(w)\le q(e,h) M$ for any $[w]\in V$ with $q(e,w)=-1$.
Therefore, for such $[w]\in V$ and $s\ge M$, we take $t=q(e,h) s-\alpha(w)\ge0$ and get
\begin{equation}\label{eq:cusp-neighborhood}
[se+w]\notin \HH(U)_{\mathrm{nr}}.
\end{equation}

We fix $\gamma$ as in Lemma~\ref{lem:irrational-axis}, with $[v_\pm]\in V$.

\begin{lemma}\label{lem:kahler-interval}
For $m\gg0$, $\ell(e,\gamma^m(e))\cap \HH(U)_{\mathrm{nr}}=\emptyset$.
\end{lemma}

\begin{proof}
We write $e=c_+v_++c_-v_-+z$ for $z\in \langle v_+,v_-\rangle_{\RR}^{\perp}\cap U_{\RR}$. Since $e,v_+,v_-$ lie on distinct rays, we have
\[c_\pm=\frac{q(e,v_\mp)}{q(v_+,v_-)}>0.\]
From the construction, we also know that $\langle v_+,v_-\rangle_{\RR}$ has signature $(1,1)$ and $\langle v_+,v_-\rangle_{\RR}^{\perp}\cap U_{\RR}$ is positive definite and $\gamma$-invariant. Consequently,
\[\lim_{m\to +\infty}\rho^{-m}\gamma^{\pm m}(e)=c_\pm v_\pm.\]
Thus, $[\gamma^{\pm m}(e)]\in V$ for $m\gg0$ and
\[\lim_{m\to +\infty}\rho^{-m}q(e,\gamma^m(e))=
c_+q(e,v_+)<0.\]

Suppose for a contradiction that there is an increasing sequence of integers $m\to\infty$ for which $\ell(e,\gamma^m(e))\cap \HH(U)_{\mathrm{nr}}\ne\emptyset$. For each such $m$, we choose
\[z_m=[s_me+\gamma^m(e)]\in \HH(U)_{\mathrm{nr}}\]
with $s_m>0$. Since $\gamma^{-m}(z_m)\in \HH(U)_{\mathrm{nr}}$ and $q(e,\gamma^{-m}(e))=q(e,\gamma^m(e))<0$,
\eqref{eq:cusp-neighborhood} applied to $e$ and $\gamma^m(e)$ gives
\[\frac{-1}{M q(e,\gamma^m(e))}<s_m<-M q(e,\gamma^m(e)).\]

Now, we choose $k_m\in\ZZ$ such that
\[1\le\lambda_m\coloneqq s_m\rho^{2k_m-m}<\rho^2.\]
After passing to a subsequence of $m$, we may assume that $\lim_{m\to \infty}\lambda_m= r\in[1,\rho^2]$. The preceding bounds imply
\[\rho^{2k_m}>\frac{-\rho^m}{M q(e,\gamma^m(e))},\qquad\rho^{2(m-k_m)}>\frac{-\rho^{m-2}}{M q(e,\gamma^m(e))}.\]
Both right-hand sides have strictly positive limits. Hence $k_m$ and $m-k_m$ are bounded below.

By $\gamma$-invariance, we have
\[\gamma^{-k_m}(z_m)=\left[\lambda_m\rho^{-k_m}\gamma^{-k_m}(e)+\rho^{-(m-k_m)}\gamma^{m-k_m}(e)\right]\in \HH(U)_{\mathrm{nr}}.\]
When $k_m+(m-k_m)=m\to\infty$, a further subsequence satisfies one of the following.
\begin{itemize}
\item $k_m,m-k_m\to\infty$. Then
\[\lim_{m\to +\infty}\gamma^{-k_m}(z_m)=[r c_-v_-+c_+v_+]\in\ell(v_-,v_+).\]

\item $k_m=k\in\ZZ$ is constant. Then
\[\lim_{m\to +\infty}\gamma^{-k_m}(z_m)=[r\rho^{-k}\gamma^{-k}(e)+c_+v_+]\in\gamma^{-k}(\ell(e,v_+)).\]

\item $m-k_m=k\in\ZZ$ is constant. Then
\[\lim_{m\to +\infty}\gamma^{-k_m}(z_m)=[r c_-v_-+\rho^{-k}\gamma^k(e)] \in\gamma^k(\ell(e,v_-)).\]
\end{itemize}
Each of the above limit is a positive combination of distinct vectors in $\partial U^-$, hence lies in $\HH(U)$. By the closedness of $\HH(U)_{\mathrm{nr}}$, each limit also lies in $\HH(U)_{\mathrm{nr}}$. But this contradicts \eqref{eq:three-geodesics} and
$\gamma(\HH(U)_{\mathrm{nr}})=\HH(U)_{\mathrm{nr}}$.
\end{proof}

\begin{proof}[Proof of Proposition~\ref{prop:finite-cover}]
Fix $m\gg0$ as in Lemma~\ref{lem:kahler-interval}, and take
\[w=-\frac{\gamma^m(e)}{q(e,\gamma^m(e))}.\]
Then $[w]\in V$ and $q(e,w)=-1$. For $a,b>0$, the 3-plane
$$W_{[ae+b\gamma^m(e)]}=\mathrm{P}(e,w)\oplus\RR(ae-b\gamma^m(e))$$ is realizable, so no MBM class of type $(1,1)$ on $X_w$ is orthogonal to $ae-b\gamma^m(e)$. Therefore, \cite[Theorem~2.10]{AV15} gives
\[\{ae-b\gamma^m(e)\mid a,b>0\}\subset\cK_{X_w}.\]
Consequently, we have $e,-\gamma^m(e)\in\overline{\cK}_{X_w},$ and $e-\gamma^m(e)\in\cK_{X_w}\cap\Lambda$. In particular, $X_w$ is projective.

By conjugating the complex structure on $X$, we have a compact hyper-K\"ahler manifold $\overline X$ on which $-e$ is semiample. Moreover, we have $(\overline X,\id)\in\mathfrak M$. Since $\gamma\in\Mon^2(X)$, we also have $(\overline X,\gamma^m)\in\mathfrak M$. Moreover, we have $q(-\gamma^m(e),-\gamma^m(h))=q(e,h)>0$ and
\[q(-\gamma^m(e),e-\gamma^m(e))=-q(e,\gamma^m(e))>0.\]
Thus the periods of $(\overline X,\gamma^m)$ and $X_w$ belong to the same component of the period domain in \cite[Definition~2.6]{SV25} for
$\RR_{>0}(-\gamma^m(e))$. Since $-\gamma^m(e)\in\overline{\cK}_{X_w}$, \cite[Theorem~3.7(v)]{SV25} implies that $-\gamma^m(e)$ is semiample on $X_w$. Thus, we get a Lagrangian fibration $g\colon X_w\to B'$ and an ample line bundle $H'$ on $B'$ such that
\[c_1(g^*(H'))\in\RR_{>0}(-\gamma^m(e)).\]

Since $c_1(\pi_w^*(L))\in\RR_{>0}e$, we obtain
\[c_1\bigl((\pi_w,g)^*(L\boxtimes H')\bigr)\in\{ae-b\gamma^m(e)\mid a,b>0\}\subset\cK_{X_w}.\]
Hence, $(\pi_w,g)$ is finite. As a general smooth fiber $A=g^{-1}(b')$ is an abelian variety, base change gives the desired finite surjection $\pi_w|_A\colon A\twoheadrightarrow B$.
\end{proof}

\begin{remark}\label{rem:singular}
Using \cite[Proposition~3.8]{syz-moduli-k3}, \cite[Proposition~5.8 and 7.7]{LehnMongardiPacienza2024}, and results in \cite{BakkerLehn2022}, Proposition~\ref{prop:finite-cover} also holds for a Lagrangian fibration from a compact singular irreducible symplectic variety $X$ with $b_2(X)\ge5$.
\end{remark}

\section{Bases of Lagrangian fibrations}\label{sec:smooth}

In this section, we prove Theorem~\ref{thm:main}.

Let $\pi\colon X\to B$ be a Lagrangian fibration from a compact hyper-K\"ahler manifold $X$, and assume that $B$ has quotient singularities. By Lemma~\ref{lem:fixed-base}, we may assume that $X$ is projective.

Let $p\colon\cB\to B$ be the canonical smooth Deligne--Mumford stack. By the universal property of the canonical stack, we have a lift
\[f\colon X\to\cB\]
with $p\circ f=\pi$. It is clear that $f$ is representable, flat, projective, and has connected geometric fibers.

\begin{lemma}\label{lem:1.1}
For every $i\ge0$, we have $R^if_*(\cO_X)\cong\Omega^i_{\cB}$.
\end{lemma}
\begin{proof}
The result follows from applying \cite[Proposition~3.6]{Ou2019} to an \'etale atlas of $\cB$.
\end{proof}

Let $\cI_0 \subset I\cB$ be an non-identity connected component of $I\cB$, with evaluation $e_0\colon\cI_0\to\cB$ and tautological automorphism $g$. We set $d=\dim(\cI_0)$. The action of $g$ induces a decomposition
\[e_0^*(\Omega^1_{\cB})=\Omega^1_{\cI_0}\oplus N_{e_0}^{\vee}.\]

\begin{lemma}\label{lem:nonid}
For every non-identity component $\cI_0\subset I\cB$, we have $c_d(T_{\cI_0})=0$ in $\CH^d(\cI_0)_{\QQ}$.
\end{lemma}
\begin{proof}
By \cite[Theorem~4.4]{Kresch}, $\cB$ is a quotient stack. We denote by $\tau_{\cB}$ the homomorphism constructed in \cite[Definition~10.8]{KrishnaSreedhar2017}. Since $IX=X$, the image of the induced morphism $If\colon IX\to I\cB$ avoids $\cI_0$. Thus, \cite[Theorem~10.11]{KrishnaSreedhar2017} gives $\tau_{\cB}([Rf_*(\cO_X)])|_{\cI_0}=0$.

Recall that for any vector bundle $E$ of rank $r$, we define a class in the K-group as $\lambda_{-1}(E)\coloneqq \sum_{i=0}^r(-1)^i[\wedge^i(E)]$. Thus, we have $[Rf_*(\cO_X)]=\lambda_{-1}(\Omega^1_{\cB})$. Therefore, we obtain 
\[\begin{aligned}
0&=\left.\tau_{\cB}([Rf_*(\cO_X)])\right|_{\cI_0}\\
&=\frac{\ch\bigl(t_g(\lambda_{-1}(e_0^*(\Omega^1_{\cB})))\bigr)}
{\ch(t_g(\lambda_{-1}(N_{e_0}^{\vee})))}
\operatorname{Td}(T_{\cI_0})\\
&=\ch\bigl(\lambda_{-1}(\Omega^1_{\cI_0})\bigr)
\operatorname{Td}(T_{\cI_0})=c_d(T_{\cI_0}),
\end{aligned}\]
where $\ch(t_g(-))$ is defined in \cite[Definition~4.8 and Section~4.3.2]{Edidin2013}, the second equality follows from \cite[Remark~10.10]{KrishnaSreedhar2017} and \cite[Theorem~5.1]{Edidin2013}, the third equality follows from the above decomposition of $e_0^*(\Omega^1_{\cB})$, and the last equality follows from \cite[Example~3.2.5]{Fulton1998}.
\end{proof}

\begin{proof}[Proof of Theorem~\ref{thm:main}]
By Proposition~\ref{prop:finite-cover} and Lemma~\ref{lem:nonid}, the result follows from Theorem~\ref{thm:smooth} and the fact that $B$ has Picard number one by \cite{Matsushita}.
\end{proof}

\end{document}